\documentclass[11pt]{article}
\usepackage{amssymb}
\usepackage{amsmath}
\usepackage{graphicx}
\usepackage{fullpage}
\usepackage{amsthm}
\usepackage{float}
\usepackage{hyperref}
\usepackage{youngtab}
\usepackage{young}
\usepackage{tikz}

\newcommand{\la}{\lambda}

\makeatletter
\newtheorem*{rep@theorem}{\rep@title}
\newcommand{\newreptheorem}[2]{%
\newenvironment{rep#1}[1]{%
\def\rep@title{#2 \ref{##1}}%
\begin{rep@theorem}}%
{\end{rep@theorem}}}
\makeatother

\newtheorem{theorem}{Theorem}
\newreptheorem{theorem}{Theorem}

\newtheorem{corollary}[theorem]{Corollary}
\newreptheorem{corollary}{Corollary}

\newtheorem{lemma}[theorem]{Lemma}
\newreptheorem{lemma}{Lemma}

\newtheorem{proposition}{Proposition}
\newtheorem*{remark}{Remark}

\usepackage{enumitem}
\usepackage{comment}

\def\bs{\bigskip}\def\ms{\medskip}
\def\N{{\mathbb N}}\def\Z{{\mathbb Z}}\def\R{{\mathbb R}}
\def\eps{\varepsilon}\def\phi{\varphi}
\def\({\mbox{$($}}\def\){\mbox{$)\/$}}

\newcommand\dyckpath[3]{
	\draw[help lines] (#1) grid +(#2,#2);
	\draw[dashed] (#1) -- +(#2,#2);
	\coordinate (prev) at (#1);
	\foreach \dir in {#3}{
		\ifnum\dir=0
		\coordinate (dep) at (1,0);
		\else
		\coordinate (dep) at (0,1);
		\fi
		\draw[line width=2pt,] (prev) -- ++(dep) coordinate (prev);
	};
}

\newcommand\labeldyck[5]{
	\coordinate (axis) at (0,0);
	\draw[help lines] (#1) grid +(#2,#3);
	\draw[dashed] (#1) -- +(#2,#3);
	\coordinate (prev) at (#1);
	\foreach \dir in {#4}{
		\ifnum\dir=0
		\coordinate (dep) at (1,0);
		\else
		\coordinate (dep) at (0,1);
		\fi
		\draw[line width=2pt] (prev) -- ++(dep) coordinate (prev);
	};
	
	\foreach \x/\y/\c in {#5}{
		\path (axis)+(\x+0.5,\y+0.5) node {$\c$};
	}
}

\title{The asymptotic normality of $(s,s+1)$-cores with distinct parts}

\author{
J\'anos Koml\'os
  \thanks{Department of Mathematics, Rutgers University}
\and Emily Sergel
    \thanks{esergel@upenn.edu. Department of Mathematics, University of Pennsylvania. Partially supported by NSF grant DMS-1603681.}
\and G\'abor Tusn\'ady
   \thanks{MTA R\'enyi Alfr\'ed Matematikai Kutat\'o Int\'ezet}
        }

\date{}

\begin{document}

\maketitle

%%%%%%%%%%%%%%%%
\begin{abstract}
Simultaneous core partitions are important objects in algebraic combinatorics. Recently there has been interest in studying the distribution of sizes among all $(s,t)$-cores for coprime $s$ and $t$. Zaleski (2017) gave strong evidence that when we restrict our attention to $(s,s+1)$-cores with distinct parts, the resulting %normalized
distribution is approximately normal. We prove his conjecture by applying the Combinatorial Central Limit Theorem and mixing the resulting normal distributions. 
\end{abstract}
%%%%%%%%%%%%%

\parindent=0pt

%%%%%%%%%%%%%%%%%%%%%%%

\section{Introduction}

\def\la{\lambda}

A partition of $n$ is a weakly decreasing sequence $\la = (\la_1 \geq \la_2 \geq \dots \geq \la_k >0)$ whose parts sum to $n$, i.e., $\la_1+\la_2+\dots+\la_k=n$. We say that $n$ is the \emph{size} of $\la$ and $k$ is its \emph{length}. For example, the partition $(4,3,3,3,2)$ has size 15 and length 5.

\ms
To each partition, we associate a diagram, known as a Ferrers diagram. The (french) Ferrers diagram of a partition $\la$ consists of boxes which are left-justified and whose $i$th row from the bottom contains $\la_i$ boxes. For example, see Figure \ref{fig:FerrersEx1}.

\begin{figure}[H]
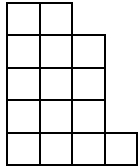

$$
\yng(2,3,3,3,4)
$$
\caption{The Ferrers diagram of the partition $(4,3,3,3,2)$.}
\label{fig:FerrersEx1}
\end{figure}

To each cell of a Ferrers diagram we associate a number known as the cell's \emph{hook length}. The hook length of a cell $c$ is the number of boxes strictly right of $c$ (known as the \emph{arm} of the cell) plus the number of boxes strictly above $c$ (the \emph{leg}) plus one. For example, the cell $c$ indicated in Figure \ref{fig:hook} has hook length 4. The cell marked $a$ is the only one in the arm and the two cells marked $\ell$ form the leg.

\begin{figure}[H]
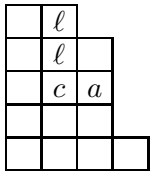

$$
\begin{Young}
& $\ell$ \cr
& $\ell$ & \cr
& $c$ & $a$ \cr
& & \cr
& & & \cr
\end{Young}
$$
\caption{The arm and leg of a cell of a Ferrers diagram.}
\label{fig:hook}
\end{figure}

For convenience, we will sometimes write the hook length of each cell into the Ferrers diagram. We say that a partition is an $s$-core if none of its cells have hook-length $s$. A partition is an $(s,t)$-core if it is simultaneously an $s$-core and a $t$-core. See Figure \ref{fig:FerrersEx}. The number of $(s,t)$-cores is finite if and only if $\gcd(s,t)=1$. Jaclyn Anderson \cite{anderson} gives a beautiful bijection between $(s,t)$-cores and certain lattice paths from $(0,0)$ to $(s,t)$ which proves this result and much more.

\begin{figure}[H]
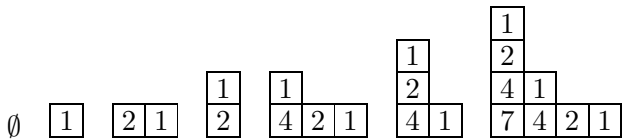

$$
\emptyset \quad \young(1) \quad \young(21) \quad \young(1,2) \quad \young(1,421) \quad \young(1,2,41) \quad \young(1,2,41,7421)
$$
\caption{The Ferrers diagram of every $(3,5)$-core with hook lengths indicated.}
\label{fig:FerrersEx}
\end{figure}

Simultaneous cores have numerous applications in algebraic combinatorics. 
For instance, Susanna Fishel and Monica Vazirani \cite{fv09,fv10} showed that when $t=ds\pm1$ for some $d \in \mathbb{N}$, they are naturally in bijection with certain regions of the $d$-Shi arrangement in type A. 
Drew Armstrong, Christopher Hanusa, and Brant Jones \cite{armstrong} extended this work to type C and related simultaneous cores to rational Catalan combinatorics. 
Purely enumerative questions have yielded deep connections as well. For instance, Armstrong \cite{armstrong} initially conjectured a simple formula for the average size of an $(s,t)$-core in 2011. Paul Johnson \cite{johnson} gave the first proof of Armstrong's conjecture by relating cores to polytopes.

\ms
As Shalosh B. Ekhad and Doron Zeilberger \cite{doron} note ``the average is just the first question one can ask about a probability distribution''. They determine the distribution obtained by fixing $t-s$, taking the size of a random $(s,t)$-core, normalizing, and letting $s \to \infty$. Surprisingly these distributions are not normal and are not known to be associated with any other combinatorial problems. However, Anthony Zaleski \cite{zaleski} gave strong experimental evidence that if $t=s+1$ and only cores with distinct parts are considered, then the resulting limit distribution is indeed normal.
We prove this in the following form.

\ms
For a positive integer $s$, let $X_s$ be the random variable given by the size of an $(s,s+1)$-core with distinct parts which is chosen uniformly at random. Let $\mu$ and $\sigma^2$ be the mean and variance of $X_s$.
Let $\Phi$ denote the standard normal distribution function.

\begin{theorem} \label{thm:main}
For all positive integers $s$,
\begin{equation} \label{eq:main}
\sup_{x\in\R}
\big|P(X_s \leq \mu+x\sigma) - \Phi(x)\big|
= O(1/\sqrt{s}).
\end{equation}
\end{theorem}
Here, and throughout the paper, the implied constants in error bounds $O(.)$ are universal constants not depending on any of our parameters.
That is, Theorem \ref{thm:main} says: There is a universal constant $C_1$ such that,
for all $s$ and $x$,
\begin{equation*}
\big|P(X_s \leq \mu+x\sigma) - \Phi(x)\big|
\leq C_1/\sqrt{s}.
\end{equation*}

\ms
% We will show that, for all $x \in \mathbb{R}$, the number of $(s,s+1)$-cores with distinct parts and size $\leq \mu+x\sigma$ is approximately $\Phi(x) F_{s+1}$. 

To prove this we introduce a new tool to this discussion: the Combinatorial Central Limit Theorem (CCLT). The original form of the CCLT is due to Wassily Hoeffding \cite{hoeffding}, but we will use the tail bounds given by Erwin Bolthausen \cite{bolt}.

\ms
Our main tools are two classical results:
Proposition \ref{proposition:matrix} on page
\pageref{proposition:matrix} (CCLT)
and Proposition \ref{proposition:pitman} on page
\pageref{proposition:pitman}
(about generating functions with only real roots).
(These two existing tools are named Propositions
and they are numbered separately.
All other statements (theorem, corollary, lemma)
are labeled in one single sequence.)

\bs
The rest of the paper is organized as follows. In Section \ref{sec:CCLT}, we review the Combinatorial Central Limit Theorem. In Section \ref{sec:fixed-k}, we prove that the distribution of size among $(s,s+1)$-cores with distinct parts is
{\em already approximately normal when the number of parts is fixed.} In Section \ref{sec:weights}, we recall that the weights needed to mix these distributions together are also approximately normal.
In Section \ref{sec:mixing} we mix these distributions together to prove Theorem \ref{thm:main}.
Section \ref{sec:lemmas} contains the proofs of some technical lemmas used in Section \ref{sec:mixing}.

\section{The Combinatorial Central Limit Theorem}\label{sec:CCLT}

\bs
Let $A=(a_{ij})$ be an $m\times m$ matrix of real numbers.
We are interested in the random sum
$$S_A = \sum_i a_{i\pi(i)}
$$
where $\pi\in S_m$ is a random permutation
of $\{1,2,\dots,m\}$ chosen uniformly from among all $m!$ permutations.
Following \cite{bolt} we write
$$
a_{i\,\cdot} = \frac1m\, \sum_j a_{ij},
  \quad a_{\cdot j} = \frac1m\, \sum_i a_{ij},
  \quad\mbox{and}\quad  a_{\cdot\cdot} = \frac1{m^2}\, \sum_{i,j} a_{ij}
$$
and set
$$
\dot{a}_{ij} = a_{ij}-a_{i\,\cdot}-a_{\cdot j}+a_{\cdot\cdot}
$$
to normalize the row- and column-sums of our matrix to 0. Furthermore, we write
$$
\mu_A = ma_{\cdot\cdot}
  \quad\mbox{and}\quad
  \sigma_A^2 = \frac1{m-1}\,\sum_{i,j}\dot{a}_{ij}^2
$$
for the mean and variance of $S_A$,
and consider the normalized sum

$$T_A = \frac{S_A - \mu_A}{\sigma_A}
= \sum_i \widehat{a}_{i\pi(i)}
$$
where
$$\widehat{a}_{ij} = \dot{a}_{ij}/\sigma_A$$

\bs\ms
The following theorem of Bolthausen \cite{bolt} gives an estimate of the remainder in the Combinatorial Central Limit Theorem.
When $A$ is of rank 1, this gives a tail bound for the classical result of Abraham Wald and Jacob Wolfowitz \cite{WW}.

\begin{proposition}\label{proposition:matrix}
% [Bolthausen 1984]
There is an absolute constant $K$ such that
for all $A$ with $\sigma_A^2>0$,

$$
\sup_t |P(T_A\leq t)-\Phi(t)|
\,\leq\,
K\sum_{i,j} |\widehat{a}_{ij}|^3/m\,.
$$

\end{proposition}

\section{Normality for a fixed number of parts} \label{sec:fixed-k}

Armin Straub \cite{straub} gave the following beautiful characterization of our chosen objects:
% \begin{proposition}  \label{proposition:straub}
{\em
A partition $\lambda$ into distinct parts is an $(s,s+1)$-core if and only if it has perimeter $\ell(\lambda)+\lambda_1 - 1 \leq s-1$.
}
% \end{proposition}

\bs
Let $k$ and $s$ be fixed non-negative integers. By the above characterization, a partition $\lambda$ consisting of $k$ distinct parts is an $(s,s+1)$-core if and only if the largest part $\lambda_1$ is at most $s-k$. We naturally associate to each such partition a vector of length $s-k$ by recording a $1$ at position $\lambda_i$ for $1 \leq i \leq k$ and $0$ elsewhere. For example, the vector $(0,1,1,0,1,0)$ corresponds to the $(9,10)$-core $(5,3,2)$.

\ms
It is now easy to see that the number of $(s,s+1)$-cores with $k$ distinct parts is just ${s-k \choose k}$. Summing shows that the total number of $(s,s+1)$-cores with any number of distinct parts is the Fibonacci number $Fib_{s+1}$. This fact was originally conjectured by Tewodros Amdeberhan \cite{teddy} and proved by Straub \cite{straub}.

\ms
We can also see that the size of the initial core is just the sum of the positions of $1$'s in the resulting vector, i.e., the inner product of this vector and $(1,2,3,\dots,s-k)$. With this rephrasing we are able to apply the CCLT: simply take the matrix $A$ to be the outer product of the vector $(1^k,0^{s-2k})$ and the vector $(1,2,3,\dots,s-k)$.

\ms
In general, suppose $A=(a_{ij})$ is an $m\times m$ rank 1 matrix, i.e., $a_{ij}=\alpha_i x_j$ for some vectors $\alpha$, $x$. 
Thus,
writing $\bar{\alpha}=(\sum \alpha_i)/m$
and $\bar{x}=(\sum x_j)/m$,
we have
\begin{equation}\label{rank-1}
\begin{split}
\dot{a}_{ij}
&=
(\alpha_i-\bar{\alpha})(x_j-\bar{x})\,,
\quad
\mu_A = m \bar{\alpha}\bar{x}  \\
\sigma_A^2
&=
\frac1{m-1}\,\sum_{i,j}\dot{a}_{ij}^2
\,=\,\frac{m^2}{m-1}\,\left(\frac1m\sum_i(\alpha_i-\bar{\alpha})^2\right)
\left(\frac1m\sum_j(x_j-\bar{x})^2\right)
\end{split}
\end{equation}

\noindent 
Let $\alpha_1=\dots=\alpha_k=1$, \,$\alpha_{k+1}=\dots=\alpha_m=0$.
Note that now $S_A$ is the sum of the elements in a random $k$-subset of the list $x_1,\dots,x_m$.
We are interested in the special case $x_i=i$ for $i=1,\dots,m$.

\begin{theorem}\label{thm:fixed-k}
For this choice of parameters
the following explicit bound holds:

\begin{equation} \label{eq:m}
\sup_{x \in \mathbb{R}} |P(T_A\leq x)-\Phi(x)|
\,\leq\,
% K\sum_{i,j} |\widehat{a}_{ij}|^3/n
\left(\frac{12m^2}{k(m-k)}\right)^{3/2}
\cdot\frac{K}{\sqrt{m}}
\end{equation}
which goes to 0 when both
$km^{-2/3}\to\infty$ and $(m-k)m^{-2/3}\to\infty$.

\end{theorem}

\begin{proof}

It is easy to see that
\begin{equation} \label{eq:mu-sig}
\bar{\alpha}=k/m,\ \bar{x}=(m+1)/2\,,
\ \mu_A=\frac{m+1}{2}\cdot k\,,
\ \sigma_A^2=\frac{m+1}{12}\,\cdot k(m-k).
\end{equation}
Using $|\dot{a}_{ij}|=|\alpha_i-\bar{\alpha}| \cdot |x_j-\bar{x}| \leq 1 \cdot m = m$,
the right-hand side in Proposition
\ref{proposition:matrix} is
\begin{align*}
K\sum_{i,j} |\widehat{a}_{ij}|^3/m
\,\leq \,\frac{Km^4}{\sigma_A^3}
\,&<
\left(\frac{12m^2}{k(m-k)}\right)^{3/2}
\cdot\frac{K}{\sqrt{m}}
\end{align*}
which goes to $0$ if $km^{-2/3}\to\infty$ and $(m-k)m^{-2/3}\to\infty$.

\end{proof}

Plugging $m=s-k$ in to \eqref{eq:m} gives the following corollary.

\begin{corollary} \label{cor:s-k}
Let $X_{s,k}$ be the random variable given by the size of an $(s,s+1)$-core with $k$ distinct parts chosen uniformly at random. Let $\mu_k$ and $\sigma_k^2$ denote the mean and variance of $X_{s,k}$, respectively. Then for any $ 0 < k < s/2$, the normalized variable $(X_{s,k} - \mu_k)/\sigma_k$ satisfies the following.\begin{equation*}
\sup_{x \in \mathbb{R}} \left| P\left(\frac{ X_{s,k} - \mu_k}{\sigma_k} \leq x \right)  - \Phi(x) \right| \leq \frac{ 12^{3/2} K (s-k)^{5/2} }{( k (s-2k) )^{3/2} }
\end{equation*}
Hence the distribution of $(X_{s,k} - \mu_k)/\sigma_k$ tends to the standard normal distribution if $s \to \infty$ and both $ks^{-2/3}\to\infty$ and $(s-2k)s^{-2/3}\to\infty$.
\end{corollary}

We will use Corollary \ref{cor:s-k} only when $s/4 \leq k \leq s/3$, in which case we obtain the bound
\begin{equation} \label{eq:useful}
\sup_{x \in \mathbb{R}} \left|
 P(X_{s,k} \leq \mu_k+x\sigma_k) - \Phi(x) \right|  < \frac{1000 K}{\sqrt{s}}.
\end{equation}

\begin{remark}
Zaleski \cite{zaleski} already noted that the generating function for $(s,s+1)$-cores with $k$ distinct parts is none other than the shifted $q$-binomial coefficient $q^{k+1 \choose 2} {s-k \choose k}_q$. It was this observation that lead us to study the distribution when $k$ is fixed. By taking $s=n+m$ and $k=m$, Corollary \ref{cor:s-k} shows that the partial sums of coefficients in the $q$-binomial coefficient ${n \choose m}_q$ are approximately normally distributed. It would be interesting to see that the distribution is also locally approximately normal.
\end{remark}

\section{The distribution of the weights} \label{sec:weights}

Ultimately %(see \eqref{eq:the-sum})
we will mix together the distributions of $X_{s,k}$ for all $k$ with $s$ fixed. Each distribution is weighted according to how many cores are being enumerated, namely $X_{s,k}$ gets weight
$$
p_k = P(W=k) = {s-k \choose k} / Fib_{s+1}.
$$
Here the random variable $W$ is the number of parts
in a random $(s,s+1)$-core with distinct parts.

\bs

The sequence ${s-k \choose k}$
appears often in combinatorics.
Its generating function
is
$$g_s(z) = \sum_{0\leq k\leq \frac{s}{2}} {s-k \choose k}z^k
= \frac{1}{\sqrt{1+4z}}
   \left(
   \left(\frac{1+\sqrt{1+4z}}{2}\right)^{s+1}
 - \left(\frac{1-\sqrt{1+4z}}{2}\right)^{s+1}
\right)
$$
--- see Concrete Mathematics \cite{concrete}
by Ronald Graham, Donald Knuth, and Oren Patashnik.
% (formula (5.74) on page 204).
By differentiating it twice, we get the moments:
$$\mu(W) := \sum_k k \, p_k = \frac{5-\sqrt{5}}{10}\cdot s\,+\,O(1)$$
and
$$\sigma^2(W) := \frac{\sqrt{5}}{25}\cdot s \,+\,O(1).$$
For convenience we write
$c_0 = (5-\sqrt{5})/10 = 0.2764..$ and $k_0 = \lfloor c_0 s \rfloor$.

\bs
There is a long history of normal approximations for finite non-negative real sequences whose generating functions have only real roots.
%%%%%%%%%%%%%%%
\begin{comment}
\bs
\centerline{(JANOS: Pitman cites \url{https://link.springer.com/content/pdf/10.1007/BF00536558.pdf} for this.)}

\bs
% (Formula (25) on page 286.)
% see also the important Proposition 1 on page 280
He attributes the local result to M.L. Platonov \cite{platonov}.  %{\bf First name?}
\end{comment}
%%%%%%%%%%%%%
The first appearance in combinatorics of a global normal law similar to \eqref{eq:global} is a result of Lawrence Harper \cite{harper} studying Stirling numbers. Harper's brilliant idea was further developed and generalized in the classical paper of Ed Bender \cite{bender}. Some important early results can be found in the paper  \cite{schoenberg} of Isaac  % Jacob
Schoenberg.
% it's about total positivity, not about the CLT
% A finite non-negative sequence is totally positive iff all zeros of its generating function are real

\ms
The following proposition is from Pitman \cite{pitman}. It says that if a polynomial $f$
with non-negative coefficients has only real zeros,
then its coefficients
are approximately normally distributed, both globally and locally.
For completeness, we cite both the global and the local versions.

\begin{proposition}\label{proposition:pitman}
Let $p_0,p_1,\dots,p_n$ be a sequence of non-negative real numbers summing to 1 with mean $\mu$ and variance $\sigma^2$. 
Let $f(x) = \sum_k p_k x^k$ be its generating function. 
Write $S_k = \sum_{i=0}^k p_i$ for the partial sums.
Assume all roots of the polynomial $f$ are real. 
Then,
\begin{equation}\label{eq:global}
\max_{0\leq k\leq n} \left|S_k
-\Phi\left(\frac{k-\mu}{\sigma}\right)\right|
< \frac{0.7975}{\sigma}
\end{equation}

and there exists a universal constant $C$ such that
\begin{equation}\label{eq:local}
\max_{0\leq k \leq n} \left| \sigma p_k
-\phi\left(\frac{k-\mu}{\sigma}\right)\right|
< \frac{C}{\sigma}.
\end{equation}

\end{proposition}

%%%%%%%%%%%%%%%
\begin{comment}
\bs
Sketch of proof.
Assume $f(x)$ is monic.
Let $-r_1,-r_2,\dots,-r_n$ be the roots of $f(x)$.
Define $\pi_i=1/(1+r_i)$.
Let $X_i$ ($1\leq i\leq n$) be independent Bernoulli random variables with
$P(X_i=1)=\pi_i$ and $P(X_i=0)=1-\pi_i$.
Write $S=\sum_{i=1}^n X_i$.

\ms
Then, clearly,
$P(S=j)=p_j$ ($1\leq j\leq n$).
Apply the classical CLTs for $S$.
(They apply as long as the $\pi_i$ are bounded away from 0 and 1.
\end{comment}
%%%%%%%%%%%%%

\ms
\begin{remark}
It is obvious that if $f$ has only real roots, then the non-negativity of the coefficients $p_0,\dots,p_n$ is equivalent to all roots of $f$ being non-positive
-- another traditional way of stating the result.
\end{remark}

\ms
Our generating function $g_s(x)$ has only real roots,
since only real numbers $z\leq -1/4$ can satisfy
$$
   \left|\, 1+\sqrt{1+4z}\,\right|
 = 
   \left|\, 1-\sqrt{1+4z}\,\right|.
$$
Hence Proposition \ref{proposition:pitman} applies to our sequence of weights $p_k={s-k \choose k}/Fib_{s+1}$ with $n=\lfloor s/2\rfloor$,
$\mu=\mu(W)$, and $\sigma=\sigma(W)$.
%%%%%%%%%%%%%%%
\begin{comment}

\bs
Assume $$|1+\sqrt{1+4z}| = |1-\sqrt{1+4z}|$$
Let $\sqrt{1+4z} = a+bi$.
Then $|(1+a)+bi| = |(1-a)-bi|$,
so $(1+a)^2 + b^2 = (1-a)^2 + b^2$, whence $a=0$.

So $\sqrt{1+4z}$ is pure imaginary,
hence $1+4z$ is real and negative or 0,
so  $z\in\R$ and $z\leq -1/4$
\end{comment}
%%%%%%%%%%%%%

\bs
The same paper \cite{pitman} (Formula (11) on page 284) contains exponential tail bounds for our weight distribution (phrased in the more general setup of so-called PF-distributions).
Plugging in our specific parameter $\mu(W)=c_0s+O(1)$,
we get the following bound: for every $\eps>0$ there is a $ \delta>0$ and a constant $C(\eps)>0$ such that

\begin{equation} \label{tail}
  \sum_{k<(c_0-\eps)s}\,p_k
\ + \sum_{k>(c_0+\eps)s}\,p_k
< C(\eps) e^{- \delta s}
\end{equation}

We will use this tail probability estimate later
with $\eps=\min\{1/3-c_0,c_0-1/4\}= 0.026..$

%%%%%%%%%%%%%%%
\begin{comment}
DON'T NEED THIS, but good to have it saved

\bs
\begin{lemma}
%  https://www.stat.berkeley.edu/~mjwain/stat210b/Chap2_TailBounds_Jan22_2015.pdf

A random variable $X$ with $\mu=EX$ is sub-Gaussian if
  there is a $\sigma>0$ such that

$$
Ee^{\lambda(X-\mu)} \leq e^{\,\sigma^2\lambda^2/2}
\mbox{\ \ for all\ \ }\lambda\in\R.
$$

that is,
$$
Ee^{\lambda X} \leq
  e^{\,\lambda\mu+\sigma^2\lambda^2/2}
\mbox{\ \ for all\ \ }\lambda\in\R.
$$

Then $X$ satisfies the the concentration inequality:

$$
P(|X-\mu|\geq t) \leq 2\,e^{-t^2/(2\sigma^2)}
\mbox{\ \ for all \ \ }t\in\R.
$$
\end{comment}
%%%%%%%%%%%%%

\section{Proof of Theorem \ref{thm:main}}
\label{sec:mixing}

\ms
Fix a positive integer $s$. Recall that $X_s$ is the random variable given by the size of an $(s,s+1)$-core with distinct parts which is chosen uniformly at random.
Zaleski \cite{zaleski} shows that the mean and variance of $X_s$ are:
\begin{equation}  \label{eq:Zaleski-moments}
\mu = \mu(X_s) = \frac{1}{10}s^2 + O(s), \quad \quad \sigma^2 = \sigma^2(X_s) = \frac{2 \sqrt{5}}{375}s^3 + O(s^2).
\end{equation}

Recall also that if $0 \leq k \leq s/2$, then $X_{s,k}$ is the random variable given by the size of an $(s,s+1)$-core with $k$ distinct parts which is chosen uniformly at random.
Hence the distribution of $X_s$ is the mixture
of the distributions of the $\lfloor s/2\rfloor+1$ individual $X_{s,k}$.

\ms
Setting $m=s-k$ in \eqref{eq:mu-sig} gives
\begin{equation}  \label{eq:our-moments}
\mu_k = \frac{1}{2}\,k\,(s+1-k), \quad \sigma_k^2 = \frac{1}{12}\,k\,(s+1-k)(s-2k).
\end{equation}

\ms
\begin{remark}
Zaleski's formulas \eqref{eq:Zaleski-moments}
could be obtained by a lengthy computation
involving the generating function $g_s(z)$, \eqref{eq:our-moments}, and 
the Pythagorean Theorem of Probability Theory
(a.k.a. the Law of Total Variance):
% which says that the total variance equals the sum of the internal variance and the external variance:
%  https://en.wikipedia.org/wiki/Law_of_total_variance
\end{remark}

\vspace{-10pt}
% \begin{equation}  \label{eq:Pythagoras}
$$
Var\big[\xi\big]
  = EVar\big[\xi|\eta\big] + 
    Var\big[E(\xi|\eta)\big].
$$

\bs
Fix $x \in \mathbb{R}$. 
Let
\begin{equation} \label{eq:the-sum}
F(x) :=P(X_s \leq  \mu + x \sigma)
= E P(X_{s,k} \leq  \mu + x \sigma).
\end{equation}

\ms
Here the expected value $E$ denotes the weighted sum
\begin{equation} \label{eq:goal}
E P(X_{s,k} \leq  \mu + x \sigma)
= \sum_{0 \leq k \leq s/2}
P(X_{s,k} \leq  \mu + x \sigma)\,p_k.
\end{equation}
For $0< k <s/2$ we can rewrite the terms
\begin{equation} % \label{eq:the-sum}
P(X_{s,k} \leq  \mu + x \sigma)
= P(X_{s,k} \leq  \mu_k + y_k \sigma_k)
=: F_k(y_k),
\end{equation}
where
\begin{equation}\label{eq:y_k}
y_k = \frac{1}{\sigma_k}\big( (\mu-\mu_k) + x\sigma \big).
\end{equation}
For $k=0$ and $k=s/2$ (when $s$ is even)
we have $\sigma_k=0$, so $y_k$ is undefined.
These at most two terms of the right-hand side of \eqref{eq:goal} have weight $1/Fib_{s+1}$ (each), so we will only work with integers $k$ with $0<k<s/2$.

\bs
Our ultimate goal is to show that $F(x)$ is approximately $\Phi(x)$ with an error bound $O(1/\sqrt{s})$ uniformly for $x\in\mathbb{R}$. We will accomplish this with a sequence of approximations $Q_1,\dots,Q_7$ and several lemmas. Each subsequent $Q$ introduces an error of only $O(1/\sqrt{s})$.
The proofs of these lemmas will be put off to Section \ref{sec:lemmas}.

\bs
Let
$$
Q_1 := \sum_{0<k<s/2} p_k F_k(y_k).
\text{\quad Then,\ \ }|F(x)-Q_1|\leq2/Fib_{s+1}.
$$

\bs
Let $I=\Z\cap(s/4,s/3)$, $J=\Z\cap(0,s/2)-I$, and
\begin{equation} \label{def:Q2}
Q_2 := \sum_{0<k<s/2}\Phi(y_k)p_k.
\end{equation}
Note that by the CCLT \eqref{eq:useful},\begin{equation}
\big|P(X_{s,k} \leq  \mu_k + y \sigma_k) - \Phi(y)\big| = O(1/\sqrt{s})
\end{equation}
uniformly for $k\in I$ and $y\in\R$.
Hence,
\begin{equation}
\big|P(X_{s,k} \leq  \mu_k + y_k \sigma_k) - \Phi(y_k)\big| = O(1/\sqrt{s})
\end{equation}
uniformly for $k \in I$ and $x\in\R$.
On the other hand, for $k \in J$ the weights $p_k$ are exponentially small in $s$ by \eqref{tail}. Since both $P(X_{s,k}\leq \mu_k+y_k \sigma_k)$ and $\Phi(y_k)$ are between 0 and 1 and the weights $p_k$ are non-negative and sum to at most 1, we have
$$
|Q_1-Q_2| = \sum_{0<k<s/2}
 p_k\cdot \big|P(X_{s,k} \leq  \mu_k + y_k \sigma_k) - \Phi(y_k)\big| = O(1/\sqrt{s}).
$$

\ms
Now we must approximate $\Phi(y_k)$ and $p_k$. We start with approximating $y_k$.
For $k\in\Z$, write $y_k^* =ax+bt_k$ where $a= \sqrt{8/5}$, $b=-\sqrt{3/5}$, and $t_k = 5^{3/4} \,(k-k_0)/\sqrt{s}$. The next lemma says that $y_k$ % (see \eqref{eq:y_k})
is well approximated by the arithmetic progression $y_k^*=ax+bt_k$ in the relevant range of $k$.
We also write $dt_k=t_k-t_{k-1} = 5^{3/4}/\sqrt{s}$.
The quantity  $dt_k$ (which is independent of $k$) will be used as a mesh size in approximating integrals. We will also see \eqref{eq:sigma-is-constant} that $\sigma_k$ is roughly constant
when $k$ is close to $k_0$.

\ms
\begin{lemma} \label{lemma:y_k-*}
For all integers $k$ with $0<k<s/2$,
\begin{equation} \label{eq:y_k-*}
|y_k-y_k^*| = \frac1{\sqrt{s}}
 \cdot O(1+|xt_k|+t_k^2).
\end{equation}
\end{lemma}

We will also show in the last section that
Lemma \ref{lemma:y_k-*} implies the following statement.

\begin{corollary} \label{corollary:y_k-*}
For all integers $k$ with $0<k<s/2$ we have

\begin{equation} \label{eq:Phi-Phi*}
|\Phi(y_k) - \Phi(y_k^*)|
= O\left(
  \frac{1}{\sqrt{s}} \, \big(1+t_k^2\big)
\right)
\end{equation}
uniformly for $x\in\R$.
\end{corollary}

Hence,
\begin{equation}
Q_2
\,= \sum_{0<k<s/2} \Phi(y_k)p_k
\,= \sum_{0<k<s/2}
\Phi(y_k^*)p_k + \frac1{\sqrt{s}}\cdot
O\left(
   \sum_{0<k<s/2} \big(1+t_k^2\big)p_k \right).
\end{equation}

\begin{lemma} \label{lemma:t_k-square}
There exists a universal constant $K_0$ such that
for all $s\in\N$,

\begin{equation} \label{eq:t_k-square}
\sum_{0 \leq k \leq s/2}(1+t_k^2)\,p_k \,\leq\, K_0.
\end{equation}
\end{lemma}
Thus,
\begin{equation}
Q_2
= \sum_{0<k<s/2}\Phi(y_k)p_k
= \sum_{0<k<s/2}\Phi(y_k^*)p_k
+ O\left( \frac1{\sqrt{s}} \right).
\end{equation}

\ms
Let
\begin{equation} \label{def:Q3}
Q_3 := \sum_{0<k<s/2}\Phi(y_k^*)p_k.
\text{\quad Then,\ \ }|Q_2-Q_3| = O(1/\sqrt{s}).
\end{equation}

\ms
It would be natural to use the local approximation \eqref{eq:local} for the weights $p_k$ at this point.
However, it would be harder to deal with the accumulation of errors. So instead we will apply the following version of summation by parts and use the % simpler
global approximation \eqref{eq:global}.

\begin{lemma} \label{lemma:summation-by-parts}
Let $m\leq n$ be integers. Suppose
$(U_k:m\leq k\leq n+1)$ and $(V_k:m-1\leq k\leq n)$
are two (finite) real sequences.
Then,

\begin{equation} \label{eq:summation-by-parts}
\sum_{k=m}^n U_k (V_k-V_{k-1})
= \sum_{k=m}^n (U_{k}-U_{k+1})V_k
\,+\,  \big[U_{n+1}V_n - U_mV_{m-1} \big].
\end{equation}

\end{lemma}

(Lemma \ref{lemma:summation-by-parts} can be verified easily by comparing the two sides term by term.)

\bs
Write
$u_k=U_k-U_{k+1}$ ($m\leq k\leq n$)
% forward difference operator $\overrightarrow{\Delta}$
and $v_k=V_k-V_{k-1}$ ($m\leq k\leq n$).
% backward difference operator $\overleftarrow{\Delta}$

Thus \eqref{eq:summation-by-parts} becomes

\begin{equation} \label{eq:SbP}
\sum_{k=m}^n U_k\, v_k
= \sum_{k=m}^n u_k\,V_k
\,+\,  \big[U_{n+1}V_n - U_mV_{m-1} \big].
\end{equation}

\ms
Note also: for all $m\leq k\leq n$,
$$U_k=U_{n+1}+\sum_{k\leq i\leq n} u_i
\mbox{\ \ and\ \ }
% \mbox{\ \,for $m\leq k\leq n$\ \,and \ }
V_k=V_{m-1}+\sum_{m\leq i\leq k} v_i.
% \mbox{\ \,(for $m\leq k\leq n$)}.
$$

\ms
\begin{corollary} \label{corollary:distance-estimate}
Let $m\leq n$ be integers. Suppose
$(U_k:m\leq k\leq n+1)$, $(U_k':m\leq k\leq n+1)$,
$(V_k:m-1\leq k\leq n)$, and $(V_k':m-1\leq k\leq n)$
are real sequences.
Define $u_k$, $u_k'$, $v_k$, $v_k'$
as in Lemma \ref{lemma:summation-by-parts}.
Write
\begin{equation}
\delta_U = \sup_{m\leq k\leq n} |U_k-U_k'|,
\quad 
\delta_V = \sup_{m\leq k\leq n} |V_k-V_k'|.
\end{equation}

\vspace{-8pt}
Then,

\vspace{-18pt}
\begin{equation}
\begin{split}
&\Big|\,
  \sum_{k=m}^n U_kv_k - \sum_{k=m}^n U_k'v_k'
\,\Big| \\
& \,\leq
\delta_U \sum |v_k'| + \delta_V \sum |u_k|
+ \left|U_{n+1}V_n-U_mV_{m-1}\right|
+ \left|U_{n+1}V_n'-U_mV_{m-1}'\right|.
\end{split}
\end{equation}

\end{corollary}

This simple corollary of Lemma \ref{lemma:summation-by-parts} will be proved in the last section.

\bs
Define
\begin{equation} \label{def:F_k^*}
F_k^* \,=\,
\begin{cases}
  1 & \hbox{if $k\leq0$,} \\
  \Phi(y_k^*)=\Phi(ax+bt_k) & \hbox{if $0<k<s/2$,} \\
  0 & \hbox{if $k\geq s/2$.}
\end{cases}
\end{equation}

Then,
\begin{equation} \label{eq:Q3}
Q_3 = \sum_{0<k<s/2} F_k^* p_k
    = \sum_{0\leq k\leq s/2}\!F_k^* p_k \,-\, p_0
    = \sum_{0\leq k\leq s/2}\!F_k^* p_k \,-\,
(1/Fib_{s+1}).
\end{equation}

\bs
Let
\begin{equation} \label{eq:Q4}
Q_4 = \sum_{0\leq k\leq s/2}\!F_k^* p_k
\end{equation}

Thus,
\begin{equation}
Q_3 \,=\, Q_4 \,-\, (1/Fib_{s+1})
    \,=\, Q_4 \,+\, O(1/\sqrt{s}).
\end{equation}

\bs
Note: The doubly infinite sequence
$(y_k^*:k\in\Z)=(ax+bt_k:k\in\Z)$
  is monotone decreasing, so $(\Phi(y_k^*):k\in\Z)$
  is monotone decreasing.
Hence $(F_k^*:k\in\Z)$ is also monotone decreasing.
Consequently, the numbers
\begin{equation} \label{def:f_k}
f_k := F_k^* - F_{k+1}^*
\quad(k\in\Z)
\end{equation}
are non-negative and add up to 1.

\bs
We apply Corollary \ref{corollary:distance-estimate}
with $m=0$,\, $n=\lfloor s/2\rfloor$,\,
$U_k=F_k^*$,\, $U_k'=\Phi(ax+bt_k)$,\,
$V_k=S_k$,\, $V_k'=\Phi(t_k)$.
Note that for us:
$U_m=U_0=1$,
$U_{n+1}=F_{\lfloor s/2\rfloor+1}^*=0$,
$V_{m-1}=S_{-1}=0$.
Hence,
\begin{equation} \label{eq:error-estimate}
\Big|\,
  \sum_{0\leq k\leq s/2} U_kv_k
\, - \sum_{0\leq k\leq s/2} U_k'v_k'
\,\Big|
\,\leq\,
\delta_U \sum |v_k'| + \delta_V \sum |u_k|
\,+\, \Phi(t_{-1}).
\end{equation}

\ms
Plugging in our values, we get
$\delta_U = 1-\Phi(ax+bt_0)$
if $s$ is odd, and when $s$ is even,
$\delta_U=
\max\{1-\Phi(ax+bt_0),\Phi(ax+bt_n)\}$.
In both cases, $\delta_U$ is exponentially small in $s$.
As far as $\delta_V$ is concerned,
\eqref{eq:global} gives
$$
\delta_V < 0.7975/\sigma(W) = O(1/\sqrt{s}).
$$

\ms
Also, both the $u_k(=f_k)$ and the
$v_k'(=\Phi(t_k)-\Phi(t_{k-1}))$
are non-negative, % and each adds up to at most 1,
hence
$$
\sum_{0\leq k\leq s/2} |u_k|
= \sum_{0\leq k\leq s/2} u_k
=U_0 - U_{n+1}
=F_0^* - F_{n+1}^* = 1-0=1
$$
and
$$
\sum_{0\leq k\leq s/2} |v_k'|
= \sum_{0\leq k\leq s/2} v_k'
= \Phi(t_n)-\Phi(t_{-1})
\leq1.
$$

Thus, \eqref{eq:error-estimate} becomes
\begin{equation} \label{eq:2323}
\Big|\,
  \sum_{0\leq k\leq s/2} U_kv_k
\, - \sum_{0\leq k\leq s/2} U_k'v_k'
\,\Big|
\,\leq\,
\delta_U + \delta_V \,+\, \Phi(t_{-1})
\,\leq\, \frac{K_1}{\sqrt{s}}
\end{equation}
for some universal constant $K_1$.

\ms
Recall that
$$
Q_4 = \sum_{0\leq k\leq s/2}\!F_k^* p_k
= \sum_{0\leq k\leq s/2} U_kv_k.
$$

\ms
Let
\begin{equation} \label{def:Q5}
Q_5 \,:= \sum_{0\leq k\leq s/2} U_k'v_k'
\, = \sum_{0\leq k\leq s/2} 
  \Phi(ax+bt_k)\big[\Phi(t_k)-\Phi(t_{k-1})\big].
\end{equation}

\bs
Thus, by \eqref{eq:2323},
$$
|Q_4-Q_5|\,\leq\,K_1/\sqrt{s}.
$$

\bs
\begin{lemma} \label{lemma:estimating-p_k*}
For all integers $k\in\Z$,
\begin{equation} \label{eq:estimating-p_k*}
\Phi(t_k) - \Phi(t_{k-1})
\,=\, \phi(t_k)dt_k
+ \frac1{\sqrt{s}}\,O\big(|\phi'(t_k)|dt_k\big)
  + O(1/s^{3/2}).
\end{equation}
\end{lemma}

\bs
Applying Lemma \ref{lemma:estimating-p_k*}, we get
\begin{equation} \label{eq:Q6}
\begin{split}
Q_5\,&=\,
\sum_{0\leq k\leq s/2}\,\Phi(ax+bt_k)\,[\Phi(t_k) - \Phi(t_{k-1})] \\
& =\,
\sum_{0\leq k\leq s/2}\,\Phi(ax+bt_k)\,\phi(t_k)\,dt_k
\,+\,\frac1{\sqrt{s}}\cdot
O\left(\sum_{0\leq k\leq s/2}|\phi'(t_k)|\,dt_k\right)
 + O(1/\sqrt{s}) \\
& =\,
\sum_{0\leq k\leq s/2}\,\Phi(ax+bt_k)\,\phi(t_k)\,dt_k
\,+\,O(1/\sqrt{s}).
\end{split}
\end{equation}
For the last line we used the fact that the $O(\sum...)$ term is a (partial) Riemann-sum for the convergent integral $\int_{-\infty}^\infty|\phi'(t)|dt$.
The bounded non-negative function $|\phi'(t)|$ is made up of four monotone pieces,
and our mesh size is $dt_k=O(1/\sqrt{s})$.

\bs
The sum in the last line of \eqref{eq:Q6} can be extended for all integers $k$ with an error of only $O(1/\sqrt{s})$. This is because
$$
\sum_{k<0}\,\Phi(ax+bt_k)\,\phi(t_k)\,dt_k
< \sum_{k<0}\,\phi(t_k)\,dt_k
$$
and the right-hand side is a Riemann sum for the function $\phi(t)$
integrated from $-\infty$ to $-5^{3/4}k_0/\sqrt{s}$.
% and the upper limit is proportional to $-\sqrt{s}$.
This integral is exponentially small in $s$.
Since on this domain $\phi(t)$ is monotone increasing and is between 0 and $1/\sqrt{2\pi}$, the Riemann sum approximation itself only introduces an error
at most $dt_k/\sqrt{2\pi} = O(1/\sqrt{s})$.
The same applies to the sum
$\sum_{k>s/2}\,\Phi(ax+bt_k)\,\phi(t_k)\,dt_k$.

\ms
Thus,
\begin{equation} % \label{def:Q5}
Q_5 = 
\sum_{k\in\Z}\,\Phi(ax+bt_k)\,\phi(t_k)\,dt_k
\,+\,O(1/\sqrt{s}).
\end{equation}

Let
\begin{equation} \label{def:Q6}
Q_6 := \sum_{k\in\Z}\,\Phi(ax+bt_k)\,\phi(t_k)\,dt_k.
\text{\quad Then,\ \ }Q_5 = Q_6 + O(1/\sqrt{s}).
\end{equation}

\ms
Define
\begin{equation} \label{def:Q7}
Q_7 := \int_{-\infty}^\infty
\Phi(ax+bt)\,\phi(t)dt.
\end{equation}

\begin{lemma} \label{lemma:total-variation}
Let $h:\R\to\R$ be a differentiable function.
Assume
$$V_h=\int_{-\infty}^\infty |\,h't)\,|\,dt < \infty.$$
Let $I_j=[\ell_j,r_j]$ \($j\in\Z$\) be a partition of $\R$ into intervals of lengths not exceeding
$\delta>0$, and let $\xi_j\in I_j$ be arbitrary points.
% It is well-known that $V_h$ is equal to the total variation of $h$. We have
Then,
$$
\left|\,
\sum_{j\in\Z} h(\xi_j)\,|I_j|
\,-\int_{-\infty}^\infty h(t)\,dt
\,\right|
\leq V_h\,\delta\,.
$$
\end{lemma}

\bs
We apply this lemma to the function
$h(t) = \Phi(ax+bt)\,\phi(t)$
with $\delta=dt_k=5^{3/4}/\sqrt{s}$.

\ms
Thus,
$
h'(t) = \phi(t)\cdot[b\,\phi(ax+bt)-t\,\Phi(ax+bt)]
$,
whence
$
|h'(t)| \leq \phi(t)\cdot(|b|+|t|)
$.

\ms
Since
$$
\int_{-\infty}^\infty |h'(t)|\,dt < \infty
$$
uniformly for $x\in\R$, by Lemma \ref{lemma:total-variation} we get
$$
Q_6 = Q_7 + O(1/\sqrt{s}).
$$

\begin{lemma}\label{lemma:mixing-normals}

Let $a$ and $b$ be real numbers.
Then for all $x\in\R$,

$$
\int_{-\infty}^\infty
\Phi(ax+bt)\,\varphi(t)dt
= \Phi\left(\frac{ax}{\sqrt{1+b^2}}\right).
$$
\end{lemma}

\bs
We apply Lemma \ref{lemma:mixing-normals} with $a=\sqrt{8/5}$ and $b=-\sqrt{3/5}$ to obtain
$$
Q_7 = \Phi(x).
$$
This completes the proof of Theorem \ref{thm:main}. Namely, we have shown that
$$F(x) \,=\,\Phi(x)\,+\,O(1/\sqrt{s})$$
uniformly in $x\in\R$.

\bs\ms
\section{Computational Proofs of the Lemmas}
\label{sec:lemmas}

\ms
\begin{replemma}{lemma:y_k-*}

For all integers $k$ with $0<k<s/2$,
\begin{equation*} 
|y_k-y_k^*| = \frac1{\sqrt{s}} \cdot O(1+|xt|+t_k^2).
\end{equation*}
\end{replemma}

\begin{proof}

Recall that $\mu_k = \frac{k(s+1-k)}{2}$, $\sigma_k^2 = \frac{k(s+1-k)(s-2k)}{12}$, $k_0 = \lfloor \frac{5-\sqrt{5}}{10} s \rfloor$. Let $D_k = k-k_0$. Then

\begin{equation}  \label{eq:sigma-is-constant}
\frac{\sigma^2_k}{\sigma_{k_0}^2} = \frac{(k_0+ D_k)(s+1-k_0- D_k)(s-2k_0-2 D_k)}{k_0(s+1-k_0)(s-2k_0)}
= 1 + O\left(\frac{ D_k}{s}\right).
\end{equation}

Therefore
$$
y_k = \frac{1}{\sigma_k} \big( (\mu-\mu_k) + x \sigma \big) = \left[ 1 + O\left( \frac{D_k}{s} \right) \right]
 \cdot \frac{1}{\sigma_{k_0}} \big( (\mu-\mu_k) + x\sigma \big).
$$
Let $q=\sigma/\sigma_{k_0}$. Then
$$
y_k = \left[ 1 + O\left( \frac{ D_k}{s} \right) \right] \cdot q \cdot \left( \frac{\mu-\mu_k}{\sigma} + x \right).
$$

Now note that
\begin{align*}
\mu_{k_0} &= \frac{1}{2}\left(\frac{5-\sqrt{5}}{10}s\right)\left(s+1-\frac{5-\sqrt{5}}{10}s\right)+O(s)\\
&=\frac{1}{2}\left(\frac{5-\sqrt{5}}{10}\right)\left(1-\frac{5-\sqrt{5}}{10}\right)s^2+O(s) \\
&= \frac{s^2}{10} +O(s)
\,=\, \mu + O(s).
\end{align*}
So
\begin{align*}
\mu-\mu_k &= \mu_{k_0} - \mu_k + O(s)\\
&= \frac{1}{2}\left( k_0(s+1-k_0) - k(s+1-k) \right) + O(s)\\
& = \frac{1}{2}(k-k_0)\Big( -s -1 +(k+k_0) \Big)+O(s)\\
& = \frac{1}{2}(k-k_0)\Big( -s -1 +(k-k_0) + 2k_0 \Big)+O(s)\\
& = \frac{1}{2}D_k \Big( -s -1 + D_k + \frac{5-\sqrt{5}}{5} \, s \Big)+O(s)\\
%& = D_k \Big( -s + D_k + s + \frac{-\sqrt{5}}{5} \, s \Big)/2+O(s)\\
%& = D_k \Big(  D_k - \frac{\sqrt{5}}{5} \, s \Big)/2+O(s)\\
&=-\frac{\sqrt{5}}{10}\cdot s D_k  + O(D_k^2) +O(s). %1+D_k^2 \geq 2*D_k
\end{align*}
(Above and below we use the obvious inequality: $2 D_k \leq  D_k^2 + 1$.)

\ms
Therefore
\begin{align*}
\frac{\mu-\mu_k}{\sigma} &= \frac{-\frac{\sqrt{5}}{10}\cdot s D_k + O(D_k^2) +O(s)}{ \sqrt{ \frac{2 \sqrt{5}}{375}} s^{3/2} \left[1 + O\left( \frac{1}{s} \right) \right] } \\
&= \sqrt{ \frac{375}{2 \sqrt{5}} } \left( -\frac{\sqrt{5}}{10}\cdot \frac{ D_k}{\sqrt{s}} +O\left(\frac{ D_k^2}{ s^{3/2} }\right) +O\left( \frac{1}{\sqrt{s}} \right) \right) \left[1 + O\left( \frac{1}{s} \right) \right] \\
&=  -\frac{3^{1/2} \ 5^{3/4}}{2^{3/2}}\cdot \frac{ D_k}{\sqrt{s}} + O \left( \frac{ D_k^2}{ s^{3/2} } \right) +O\left( \frac{1}{\sqrt{s}} \right).
\end{align*}

Finally, setting $t_k=5^{3/4} D_k / \sqrt{s}$
and using $|D_k|\leq s$ gives
\begin{align*}
y_k &= \left[ 1 + O\left( \frac{ D_k}{s} \right) \right] \cdot q \cdot \left( \frac{\mu-\mu_k}{\sigma} + x \right)\\ 
&= \left[ 1 + O\left( \frac{ t_k}{\sqrt{s}} \right) \right] \cdot q \cdot \left(x -\sqrt{\frac{3}{8}} \, t_k + O \left( \frac{ t_k^2}{ \sqrt{s} } \right) +O\left( \frac{1}{\sqrt{s}} \right) \right)\\
&= q \cdot \left(x - \sqrt{ \frac{3}{8} } \, t_k\right) + \frac{1}{\sqrt{s}} \cdot O \left( 1+ |xt_k|+t_k^2 \right).
\end{align*}

But $q$ is essentially a constant. That is,
\begin{align*}
q^2 = \frac{\sigma^2}{\sigma_{k_0}^2} &= \frac{  \frac{2\sqrt{5}}{375} s^3 + O(s^2)  }{   \frac{1}{12}k_0(s+1-k_0)(s-2k_0) + O(s^2)  } \\
&= \frac{  \frac{2\sqrt{5}}{375} s^3 + O(s^2)  }{   \frac{1}{12}c_0(1-c_0)(1-2c_0)s^3 \left[ 1+ O\left(\frac{1}{s}\right) \right]  }\\
% check on MAPLE: constants give 8/5
&= \frac{8}{5}+ O\left(\frac{1}{s}\right).
\end{align*}
So \,$q = \sqrt{8/5} + O(1/s)$. Therefore
\begin{align*}
y_k 
&= \left(\sqrt{\frac{8}{5}} \, x - \sqrt{ \frac{3}{5} } \, t_k \right)
 + \frac{1}{\sqrt{s}} \cdot O\left( 1+|xt|+t_k^2 \right)
 = y_k^*
 + \frac{1}{\sqrt{s}} \cdot O\left( 1+|xt|+t_k^2
 \right).
\end{align*}
\end{proof}

\ms
\begin{repcorollary}{corollary:y_k-*}

For all integers $k$ with $0<k<s/2$ we have

\begin{equation*}
|\Phi(y_k) - \Phi(y_k^*)|
= O\left(
  \frac{1}{\sqrt{s}} \, \big(1+t_k^2\big)
\right)
\end{equation*}
uniformly for $x\in\R$.

\end{repcorollary}

\begin{proof}  \mbox{}

\ms
Let $K_2$ be the implied constant in \eqref{eq:y_k-*}. % We assume $K_2\geq1$.
Let $\eps_1=\sqrt{2/3}$,
$x_0=16K_2/a$, and
$s_0=(8K_2/a)^4$.

\ms
{\em Special case I:}\ \, $|t_k|\geq s^{1/4}$.

\ms
Then $1+t_k^2>s^{1/2}$,
so $\frac{1}{\sqrt{s}}(1+t_k^2)>1$.
Hence \eqref{eq:Phi-Phi*} is automatically true (independent of the value of $x$).

\ms
{\em Special case II:}\ \, $|t_k|\geq \eps_1|x|$.

\ms
Then $1+|xt_k|+t_k^2 \leq 1+(\frac1{\eps_1}+1)t_k^2
< 3(1+t_k^2)$.

\ms
{\em Special case III:}\ \, $|x|\leq x_0$.

\ms
Then $1+|xt_k|+t_k^2
\leq 1+x_0|t_k|+t_k^2\leq(1+x_0/2)(1+t_k^2)
=O(1+t_k^2)$.

\bs
For the rest of this proof we will assume
$k$ is an integer with $0<k<s/2$ satisfying:
\begin{equation*}
x>x_0,\quad |t_k|<\eps_1|x|,\quad\mbox{and}\ \ |t_k|<s^{1/4}.
\end{equation*}

\ms
We will first show that both $y_k$ and $y_k^*$ are between $\frac14\,ax$ and $\frac74\,ax$. This will allow us to apply the Mean Value Theorem to prove the corollary.

\ms
Recall that $a=\sqrt{8/5}$, $b=-\sqrt{3/5}$,
and $t_k = 5^{3/4} \,(k-k_0)/\sqrt{s}$.
Thus,
$$
|bt_k|=\sqrt{3/5}\,|t_k| < \sqrt{3/5}\,\eps_1|x|
% = \big(\eps_1\sqrt{3/8}\big)(a|x|)
=\frac12|ax|.
$$

\ms
Consequently,
\vspace{-10pt}
$$
\,y_k^*=ax+bt_k \text{\ is between\ } \frac12\,ax
\text{\ and\ } \frac32\,ax,
\text{\ whence\ } |y_k^*| > \frac12\,a|x|.
$$

\bs
Now we estimate $y_k$:

$$
|y_k^*-y_k|\leq \frac{K_2}{\sqrt{s}}
\cdot(1+|xt_k|+t_k^2)
= \frac{K_2}{\sqrt{s}}\cdot(1+t_k^2)
+ \frac{K_2}{\sqrt{s}}\cdot|xt_k|.
$$
The first term on the right-hand side is estimated as
$$   \frac{K_2}{\sqrt{s}}\,(1+t_k^2)
<    \frac{K_2}{\sqrt{s}}\,(1+s^{1/2})
=    K_2\,(1+s^{-1/2})
\leq 2K_2
\leq\frac18\,a|x|
$$
for $x\geq x_0$.

\ms
For the second term we have
$$   \frac{K_2}{\sqrt{s}}\cdot|xt_k|
<    \frac{K_2}{\sqrt{s}}\cdot |x|s^{1/4}
=    \frac{K_2}{s^{1/4}}\cdot |x|
\leq\frac18\,a|x|
$$
for $s\geq s_0$.

\bs
Consequently,
\vspace{-10pt}
$$|y_k^*-y_k| < \frac14\,a|x|,
\text{\ and thus $y_k$ is between $\frac14\,ax$
and $\frac74\,ax$ as desired.}
$$

\bs
By the Mean Value Theorem, there is a $\xi$
between $y_k$ and $y_k^*$ such that
$\Phi(y_k^*)-\Phi(y_k) = \phi(\xi)\,(y_k^*-y_k)$.
As we showed above,
$\xi$ is between $\frac14\,ax$ and $\frac74\,ax$,
and hence
$$
|\xi| > \frac14\,a|x|
> \frac{a}{4\eps_1}\,|t_k|.
$$
Consequently, since $\phi$ is monotone,

\begin{equation*}
\phi(\xi) = \phi(|\xi|)
< \phi\left(\frac14\,a|x|\right)
\quad\text{and}\quad
\phi(\xi) < \phi\left(\frac{a}{4\eps_1}\,|t_k|\right).
\end{equation*}

\bs
We obtain:
\begin{equation*}
\begin{split}
|\Phi(y_k^*)-\Phi(y_k)| & = \phi(\xi)\,|y_k^*-y_k|
\leq \phi(\xi)\,\frac{K_2}{\sqrt{s}}\cdot (1+|xt_k|+t_k^2) \\
& < \frac{K_2}{\sqrt{s}}
\cdot\left[\,(1+t_k^2)\,
\phi\left(\frac{a}{4\eps_1}\,|t_k|\right)
+ \eps_1x^2\,\phi\left(\frac14\,a|x|\right)
\right].
\end{split}
\end{equation*}

\ms
Since the quantity in square brackets
is bounded uniformly in $k\in\mathbb{Z}$ and $x\in\R$,
Corollary \ref{corollary:y_k-*} is proved.

\end{proof}

\ms
\begin{replemma}{lemma:t_k-square}

There exists a universal constant $K_0$ such that
for all $s\in\N$,
\begin{equation*}
\sum _{0\leq k\leq s/2}
(1+t_k^2)\,p_k \,\leq\, K_0.
\end{equation*}

\end{replemma}

\begin{proof}
By the definition of $t_k$, we have
\begin{equation*}
t_k^2 = \frac{5^{3/2}}{s}\, (k-k_0)^2
\leq \frac{25}{s}\cdot
  \left[ (k-\mu(W))^2 + (\mu(W)-k_0)^2 \right]
= \frac{25}{s}\, (k-\mu(W))^2 +O(1/s).
\end{equation*}
Here we used
$(\alpha-\gamma)^2\,\leq\,2[(\alpha-\beta)^2 + (\beta-\gamma)^2]$.
Hence,

\begin{equation*}
\sum_{0\leq k\leq s/2}t_k^2\,p_k
\leq
\frac{25}{s} \,
 \sum_{0\leq k\leq s/2}
(k-\mu(W))^2p_k +O(1)
= 25\cdot \frac{\sigma^2(W)}{s}  +O(1)
= O(1)
\end{equation*}
(where, as always, $O(1)$ is independent of $s$).

\end{proof}

\ms
\begin{repcorollary}{corollary:distance-estimate}

For sequences $U,U',V,V'$ as before,

\vspace{-10pt}
\begin{equation*}
\Big|\,
  \sum_{k=m}^n U_kv_k - \sum_{k=m}^n U_k'v_k'
\,\Big|
 \,\leq
\delta_U \sum |v_k'| + \delta_V \sum |u_k|
+ \left|U_{n+1}V_n-U_mV_{m-1}\right|
+ \left|U_{n+1}V_n'-U_mV_{m-1}'\right|.
\end{equation*}

\end{repcorollary}

\begin{proof}

We start with the following four identities,
the non-trivial two of which follow from applying Lemma \ref{lemma:summation-by-parts} twice.

\begin{equation*}
\sum_{k=m}^n U_kv_k - \sum_{k=m}^n u_kV_k 
\,=\, \left[U_{n+1}V_n-U_mV_{m-1}\right].
\end{equation*}

\begin{equation*}
\sum_{k=m}^n u_kV_k - \sum_{k=m}^n u_kV_k'
\,=\, \sum_{k=m}^n u_k(V_k-V_k').
\end{equation*}

\begin{equation*}
\sum_{k=m}^n u_kV_k' - \sum_{k=m}^n U_kv_k'
\,=\, -  \left[ U_{n+1}V_n'-U_mV_{m-1}' \right].
\end{equation*}

\begin{equation*}
\sum_{k=m}^n U_kv_k' - \sum_{k=m}^n U_k'v_k'
\,=\, \sum_{k=m}^n (U_k-U_k')v_k'.
\end{equation*}

\ms
Adding up these four identities we get
\begin{equation*}
\begin{split}
&  \sum_{k=m}^n U_kv_k - \sum_{k=m}^n U_k'v_k' \\
&= \sum_{k=m}^n (U_k-U_k')v_k'
+  \sum_{k=m}^n u_k(V_k-V_k')
+  \left[ U_{n+1}V_n-U_mV_{m-1} \right]
-  \left[ U_{n+1}V_n'-U_mV_{m-1}' \right],
\end{split}
\end{equation*}
from which Corollary \ref{corollary:distance-estimate} follows.

\end{proof}

\ms
\begin{replemma}{lemma:estimating-p_k*}

For all integers $k\in\Z$,
$$
\Phi(t_k) - \Phi(t_{k-1})
\,=\, \phi(t_k)dt_k
+ \frac1{\sqrt{s}}\,O(|\phi'(t_k)|dt_k)
  + O(1/s^{3/2})
$$
where $dt_k=5^{3/4}/\sqrt{s}$.

\end{replemma}

\begin{proof}
Let $k\in\Z$. There exists a $\xi_k$ with $t_{k-1}<\xi_k<t_k$ such that

\begin{align*}
\Phi(t_k)-\Phi(t_{k-1})
   &= \phi(t_k) (t_k-t_{k-1}) - \frac12 \phi'(t_k)(t_k-t_{k-1})^2
 + \frac16\phi''(\xi_k)(t_k-t_{k-1})^3 \\
   &= \phi(t_k) dt_k - \frac12 \phi'(t_k)(dt_k)^2
 + \frac16\phi''(\xi_k)(dt_k)^3 \\
   & = \phi(t_k)dt_k
  \,+\,\frac1{\sqrt{s}}\,O(|\phi'(t_k)|dt_k)
  + O(1/s^{3/2}).
\end{align*}

\end{proof}

\ms
\begin{replemma}{lemma:total-variation}

Let $h:\R\to\R$ be a differentiable function.
Assume
$V_h=\int_{-\infty}^\infty |\,h't)\,|\,dt < \infty$.
Let $I_j=[\ell_j,r_j]$ \($j\in\Z$\) be a partition of $\R$ into intervals of lengths not exceeding
$\delta>0$, and let $\xi_j\in I_j$ be arbitrary points.
% It is well-known that $V_h$ is equal to the total variation of $h$. We have
Then,
$$
\left|\,
\sum_{j\in\Z} h(\xi_j)\,|I_j|
\,-\int_{-\infty}^\infty h(t)\,dt
\,\right|
\leq V_h\,\delta\,.
$$

\end{replemma}

\begin{proof}

While the statement is known in the context of total variations of functions, we give, for completeness, a simple direct proof by applying the bounded version below on each individual interval $I_j$.

% \begin{lemma} \label{lemma:emily}
\newtheorem*{observation}{Observation}
\newtheorem*{bounded-version}{Bounded version}

\begin{observation} \label{lemma:emily}
Let $h$ be a differentiable function on a closed interval $I=[a,b]$ \($a<b$\). Then,

\begin{equation*}
|h(b)-h(a)| \leq \int_a^b |h'(t)|\,dt.
\end{equation*}
\end{observation}

\bs

Indeed, by the Fundamental Theorem of Calculus,
\begin{equation*}
\big|h(b)-h(a)\big|
=  \left|\int_a^b h'(t)\,dt \right|
\leq \int_a^b |h'(t)|\,dt.
\end{equation*}

\ms
\begin{bounded-version}
Let $h$ be differentiable on a closed bounded interval $I=[a,b]$ \($a<b$\).
Let $\xi\in I$ be arbitrary. Then,

\begin{equation*}
D\,:=\,
  \left|\,
    h(\xi)\cdot(b-a) \,-\, \int_a^b h(t)\,dt
  \,\right|
\,\leq\,  (b-a) \int_a^b |h'(t)|\,dt.
\end{equation*}
\end{bounded-version}

\bs
Indeed, since $h$ is continuous on $I$,
there exists an $\eta\in I$ such that

$$
\int_a^b h(t)\,dt \,=\, h(\eta)\cdot(b-a).
$$

\ms
Assume (WLOG) that $\eta\leq\xi$.
Then, by the Observation above,
% Lemma \ref{lemma:emily},
$$
D \,=\,(b-a)\cdot \big|h(\xi)-h(\eta)\big|
\,\leq\, (b-a) \int_\eta^\xi |h'(t)|\,dt
\,\leq\, (b-a) \int_a^b |h'(t)|\,dt.
$$

\end{proof}

\ms
\begin{replemma}{lemma:mixing-normals}

Let $a$ and $b$ be real numbers.
Then for all $x\in\R$,
$$
\int_{-\infty}^\infty
\Phi(ax+bt)\,\varphi(t)dt
= \Phi\left(\frac{ax}{\sqrt{1+b^2}}\right).
$$

\end{replemma}

\begin{proof}
One could compute the two-dimensional integral corresponding to the left hand side.
We present instead a simple probabilistic proof.
We write $E$ for expected value.

\ms
Let $Z_1$ and $Z_2$ be independent standard normal variables. Define $Z_3 = Z_1-bZ_2$.
Then $Z_3$ is a normal random variable
with 0 expectation and variance $1+b^2$.
We then have

\begin{equation*}
 \begin{split}
\int_{-\infty}^\infty
 \Phi(ax+bt)\,\varphi(t)dt
& = E \Phi(ax+bZ_2)
 = E P(Z_1 \leq ax+bZ_2) \\
& = E P(Z_3 \leq ax)
= \Phi\left(\frac{ax}{\sqrt{1+b^2}}\right).
 \end{split}
\end{equation*}
%% since $\Phi(ax+bZ_2)$ is a function of $Z_2$.
\end{proof}

%\nocite{*} %imports citations that you don't actually cite
\bibliography{cores}
\bibliographystyle{alpha} %[Auth99]

\end{document}